\documentclass{article}
\usepackage{amssymb}
\usepackage{amsfonts}
\usepackage{amsmath}
\usepackage{dsfont}
\usepackage{setspace}
\usepackage{hyperref}
\usepackage{graphicx}

\newtheorem{theorem}{Theorem}[section]

\newtheorem{definition}{Definition}[section]
\newtheorem{lemma}{Lemma}[section]
\newtheorem{proposition}[theorem]{Proposition}
\newtheorem{remark}{Remark}[section]
\newenvironment{proof}[1][Proof]{\noindent\textbf{#1.} }{\ \rule{0.5em}{0.5em}}
\providecommand{\keywords}[1]{\textit{Keywords and phrases - } #1}
\providecommand{\AMS}[1]{\textit{AMS 2010 Subject Classification - } #1}
\input{tcilatex}
\onehalfspace

\begin{document}

\title{Quantitative central limit theorems for Mexican needlet coefficients on circular Poisson fields}
\author{Claudio Durastanti\thanks{%
e-mail address: durastan@mat.uniroma2.it} \thanks{%
This research is supported by European Research Council Grant n.277742 PASCAL%
} \\
%EndAName
University of Tor Vergata, Rome}
\maketitle

\begin{abstract}
The aim of this paper is to establish rates of convergence to Gaussianity for wavelet coefficients on circular Poisson random fields. This result is established by
using the Stein-Malliavin techniques introduced by Peccati and Zheng (2011)
and the concentration properties of so-called Mexican needlets on the circle.
\end{abstract}
\AMS{60F05, 60G60, 62E20, 62G20\\}

\section{Introduction}

\subsection{Motivations}

This work is concerned with the study of quantitative central limit theorems
for linear statistics based on wavelet coefficients computed on circular
Poisson random fields. In particular, we are referring to the very
remarkable advances provided in this area by the combination of two
probabilistic methods, the Malliavin calculus of variations and the Stein's
method of approximations. The interaction of these methods is successfully
applied to exploit rates of convergence of the asymptotic normal
approximation for functionals of Gaussian random measures (see \cite{np, nup}%
), for functionals of general Poisson random measures (cfr. \cite{peccati1,
peccati2} and \cite{bp}) and, more recently, to fix convergence criteria
from the point of view of spectral theory of general Markov diffusion
generators (see \cite{campese, ledoux}). These results have being used in
a growing range of applications, see for instance \cite{hwz, lwx, yhwl}: we
mention as textbook reference \cite{npbook} while the webpage
http://www.iecn.u-nancy.fr/\hbox{$\scriptstyle\mathtt{\sim}$}%
nourdin/steinmalliavin.htm contains updates on the researches on this field.

The main object of our investigation concerns the application of these
techniques to the framework of wavelet coefficients computed over samples of
circular data. These data correspond to the measures of angles labelled by a
given origin, i. e. a starting point on $\mathbb{S}^{1}$, and a given
positive direction. From the theoretical point of view, the datasets on $%
\mathbb{S}^{1}$ are characterized by a periodicity property, with period $%
2\pi $. As consequence, a lot of interest has been recently raised by the
development of statistical methods on the circle, also in view of their
applications in many different sciences, as for instance geophysics,
cosmology, oceanography and engineering. A complete overview on this topic
can be found, for instance, in the textbooks \cite{bhatta, fisher,
rao,silverman}. Some more recent applications can be found in \cite{alial,
dimarzio, durastanti3, klemklem, wu}. In recent years (cfr. \cite{ant1}%
), the literature concerning the unit $q$-dimensional sphere has made an
extensive use of the construction of second-generation wavelets on the
sphere, the so-called spherical needlets. Introduced in the literature by 
\cite{npw1, npw2}, spherical needlets are characterized by some main
properties which makes them an excellent tool for statistical analysis, such
as their concentration in both Fourier and space domains. Spherical
needlets and some extensions were successfully applied to a large set of
statistical problems, see for instance \cite{bkmpAoS, bkmpAoSb, cm, dlm}.

Some assessments of quantitative Berry-Esseen bounds for statistics related
to the needlet framework are already present in the literature: the
multidimensional normal approximation of linear and nonlinear statistics
based on needlet coefficients evaluated either on Poisson field or on
vectors of i.i.d. observations over the sphere were studied, respectively,
in \cite{dmp} and in \cite{bdmp}. The wavelets taken into account here are
instead the so-called Mexican needlets, built over a general compact
manifold by D. Geller and A. Mayeli in \cite{gm0, gm1, gm2, gm3}. A Mexican
needlet $\psi _{jq;s}$ is indexed by the shape parameter $s$, the resolution
level $j$ and by $q$, which indicates the region $E_{jq}\subset \mathbb{S}%
^{1}$ on which the needlet is consistently different from $0$. It can be
roughly thought as a product between basis elements computed on a suitable
set of points $x_{jq}\in E_{jq}$ and the Schwarz function $w_{s}:\mathbb{R}%
\mapsto \mathbb{R}_{+}$. The function $w_{s}$ leads to an exponential
concentration property in the spatial domain, stronger than the localization
related to standard needlets\ (cfr. also \cite{durastanti1, durastanti3}).
Moreover, while the standard needlets are defined over a set of exact
cubature points and weights (cfr. \cite{npw1}) to have a tight frame, the
Theorem 2.2 in \cite{gm2} establishes that the frame obtained by the Mexican
needlets built over a set of points under some weaker conditions (see \cite%
{gm2} and Section \ref{subharmonic} below) is nearly-tight. Various examples
of statistical applications of Mexican needlets can be found, for instance,
in \cite{scodeller, lanmar2, mayeli, durastanti3, dll}.

\subsection{Main results}

This work is concerned with quantitative rates of
convergence to Gaussianity of Mexican needlet coefficients
sampled over Poisson processes. Consider a sequence of independent and
identically distributed random variables $\left\{ X_{i},i\geq 1\right\} $,
taking values over $\mathbb{S}^{1}$ so that $b_{jq;s}:=\mathbb{E}\left[ \psi
_{jq;s}\left( X_{1}\right) \right] $ and $\sigma _{jq;s}:=\mathbb{E}\left[
\psi _{jq;s}^{2}\left( X_{1}\right) \right] $. Let us consider the
independent Poisson process $\left\{ N_{t}:t\geq 0\right\} $ on $\mathbb{R}$
with parameter $R_{t}$, which is monotonically increasing with $t$. Our
purpose is to establish conditions over the three sequences $\left\{
j=j_{t}:t\geq 1\right\} $, $\left\{ q=q_{t}:t\geq 1\right\} $ and $\left\{
R_{t}:t\geq 1\right\} $ so that, in the sense of the distance $d_{2}$, the $%
d_{t}$-dimensional vector $Y_{t}=\left( Y_{t,1},...,Y_{t,d_{t}}\right) $,
where%
\begin{equation}
Y_{t,i}:=\frac{1}{\sqrt{R_{t}}}\sum_{i=1}^{R_{t}}\left( \frac{\psi
_{j_{t}q_{t};s}\left( X_{i}\right) -b_{j_{t}q_{t};s}}{\sigma _{j_{t}q_{t};s}}%
\right) \text{ ,}  \label{Ycomponent}
\end{equation}%
is asymptotically close to a Gaussian $d_{t}$-dimensional standard Gaussian
random vector $Z_{d_{t}}$. These results are stated in Theorems \ref%
{theorempeccatiunivariate} and \ref{theorempeccati}. Similar results were
obtained on $\mathbb{S}^{2}$ by using standard needlets, cfr. \cite{dmp}.
Furthermore, we study also the so-called 'de-Poissonized' case, where the
data are i.i.d. over $\mathbb{S}^{1}$ and for which we will establish a
quantitative central limit theorem, cfr. Proposition \ref{depoisson}. We
will also propose a case study concerning the nonparametric density
estimation, which can be considered as a completion of our previous work 
\cite{durastanti3}.

From the technical point of view, the proofs of the main theorems follow
strictly the guidelines driven for this kind of application by \cite{dmp}.
On the other hand, the ancillary results concerning the asymptotic behaviour
of the covariance matrix of Mexican needlet coefficients sampled on Poisson
random processes are also of interest. They are obtained by using the
localization property of Mexican circular needlets developed in \cite%
{durastanti3}.

\subsection{Plan of the paper}

The Section \ref{subharmonic} introduces some preliminary notions, such as
Mexican needlets systems and their properties and the general results on
Stein-Malliavin bounds from \cite{peccati1, peccati2}. The Section \ref%
{secmain} presents the statement of our main results on the rate of
convergence in the Gaussian approximation of linear statistics of Mexican
needlet coefficients by Stein-Malliavin techniques. The Section \ref%
{secpeccati} is concerned with the proofs of the main theorems and of the
auxiliary results. The Section \ref{secapplication} studies an application
to the framework of the nonparametric density estimation of the results in
the Theorem \ref{theorempeccati}. The Section \ref{secsimulation} contains
some numerical evidence.

\section{Preliminary results\label{subharmonic}}

\subsection{The Mexican needlet framework}

In this section, we will introduce the construction of the Mexican needlets
over the unit circle $\mathbb{S}^{1}$: these wavelets were introduced
in the literature by D. Geller and A. Mayeli in \cite{gm0, gm1, gm2, gm3}.
We will start with a quick overview on the Fourier analysis over the circle:
let $L^{2}\left( \mathbb{S}^{1}\right) \equiv L^{2}\left( \mathbb{S}%
^{1},d\rho \right) $ be the space of square integrable functions over the
circle with respect to the uniform Lebesgue measure $\rho \left( d\theta
\right) :=\left( 2\pi \right) ^{-1}d\theta $. As well-known in the
literature, the set of functions $\left\{ u_{k}\left( \theta \right) ,\theta
\in \mathbb{S}^{1},k\in \mathbb{Z}\right\} $, $u_{k}\left( x\right) =\exp
\left( ik\theta \right) $, is an orthonormal basis over $\mathbb{S}^{1}$.
For $f\in L^{2}\left( \mathbb{S}^{1}\right) $, we define the Fourier
transform as%
\begin{equation*}
a_{k}:=\frac{1}{2\pi }\int_{0}^{2\pi }f\left( \theta \right) \overline{%
u_{k}\left( \theta \right) }d\theta \text{ ,}
\end{equation*}%
while the Fourier expansion is given by%
\begin{equation}
f\left( \theta \right) =\sum_{k\in \mathbb{Z}}a_{k}u_{k}\left( \theta
\right) \text{ , }\theta \in \mathbb{S}^{1}\text{ .}
\label{fourier expansion}
\end{equation}%
Observe that $u_{k}\ $is the eigenfunction of the circular Laplacian $\Delta 
$ corresponding to eigenvalue $-k^{2}$, further details can be found in the
textbook\cite{steinweiss}, see also \cite{marpecbook}.

\begin{figure}[tbp]
\centering
\includegraphics[width=\textwidth]{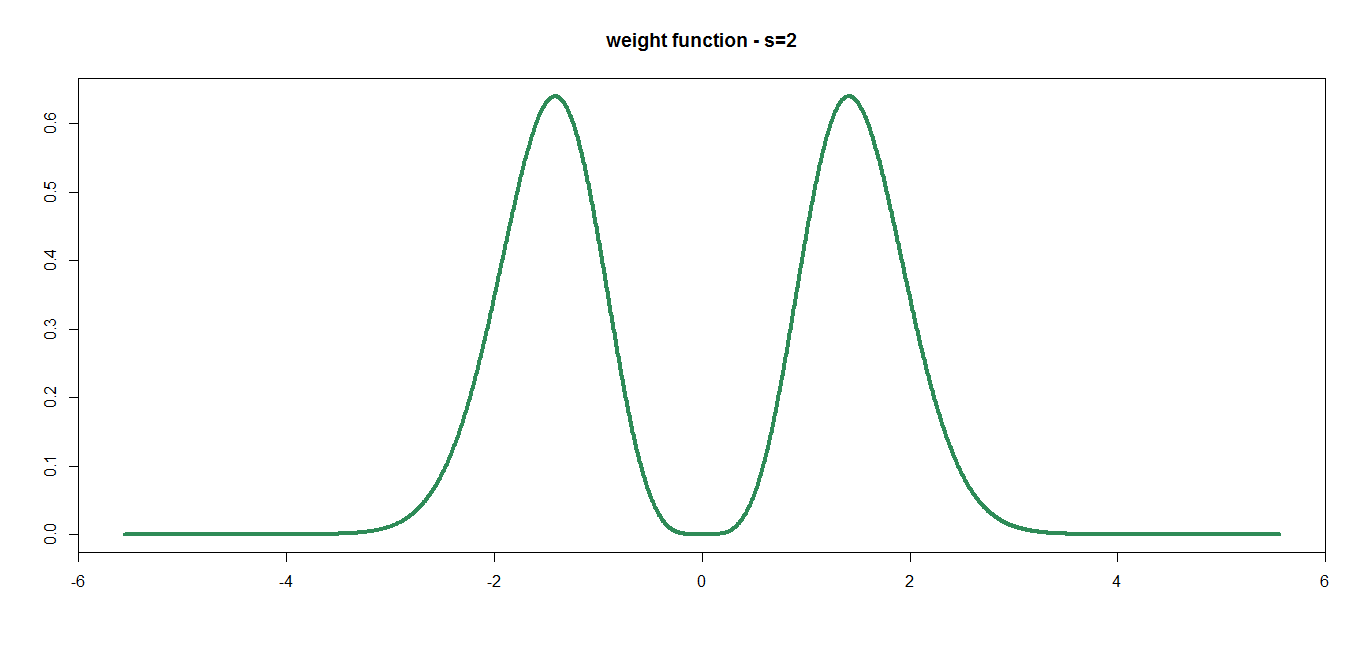}
\caption{the weight function $w_{s}$ for $s=2$.}
\label{fig:1}
\end{figure}

Consider now the function $w_{s}:\mathbb{R}\rightarrow \mathbb{R}_{+}
$, named weight function (cfr. Figure \ref{fig:1}) and defined as%
\begin{equation}
w_{s}\left( x\right) :=x^{s}\exp \left( -x\right) \text{ , }x\in \mathbb{R}%
\text{ .}  \label{weight}
\end{equation}%
Following \cite{gm0}, from the Calderon formula and for $t\in \mathbb{R}_{+}$%
, we define%
\begin{equation*}
e_{s}:=\int_{0}^{\infty }\left\vert w_{s}\left( tx\right) \right\vert ^{2}%
\frac{dx}{x}=\frac{\Gamma \left( 2s\right) }{2^{2s}}\text{ ;}
\end{equation*}%
on the other hand, (see \cite{gm2}) fixing the scale parameter $B>1$, from
the Daubechies' Condition it follows that%
\begin{equation*}
\Lambda _{B,s}m_{B}\leq \sum_{j=-\infty }^{\infty }\left\vert w_{s}\left(
tB^{-2j}\right) \right\vert ^{2}\leq \Lambda _{B,s}M_{B}\text{ ,}
\end{equation*}%
where $\Lambda _{B,s}=e_{s}\left( 2\log B\right) ^{-1}$, $M_{B}=\left(
1+O_{B}\left( \left\vert B-1\right\vert ^{2}\log \left\vert B-1\right\vert
\right) \right) $ and $m_{B}=\left( 1-O_{B}\left( \left\vert B-1\right\vert
^{2}\log \left\vert B-1\right\vert \right) \right) $.

Fixed the resolution level $j\in \mathbb{Z}$, consider a partition of $%
\mathbb{S}^{1}$ $\left\{ E_{jq}:q=1,...,Q_{j}\right\} $ such that for any $%
q_{1}\neq q_{2}$, $E_{jq_{1}}\cap E_{jq_{2}}=\varnothing $. The region $%
E_{jk}$ can be described in terms of the couple $\left( \lambda
_{jq},x_{jq}\right) $: the positive constant $\lambda _{jq}:=\rho \left(
E_{jq}\right) $ is the length of $E_{jq}$, and $x_{jq}\in E_{jq}$ is a
generic point belonging to $E_{jq}$. For the sake of simplicity, from now on
we will consider $x_{jq}$ as the midpoint of the segment of arc $E_{jq}$.
Fixing the shape and the scale parameters $s\in \mathbb{N}$ and $B>1$ (cfr.
Figure \ref{fig:2}), the circular Mexican needlet $\psi _{jq;s}:\mathbb{S}%
^{1}\mapsto \mathbb{C}$ is defined as%
\begin{eqnarray}
\psi _{jq;s}\left( \theta \right)  &:&=\sqrt{\lambda _{jq}}\sum_{k=-\infty
}^{\infty }w_{s}\left( \left( B^{-j}k\right) ^{2}\right) \overline{%
u_{k}\left( x_{jq}\right) }u_{k}\left( \theta \right)   \notag \\
&=&\sqrt{\lambda _{jq}}\sum_{k=-\infty }^{\infty }w_{s}\left( \left(
B^{-j}k\right) ^{2}\right) \exp \left( ik\left( \theta -x_{ju}\right)
\right) \text{, }\theta \in \mathbb{S}^{1}\text{ .}  \label{needletdef}
\end{eqnarray}

\begin{figure}[tbp]
\centering
\includegraphics[width=\textwidth]{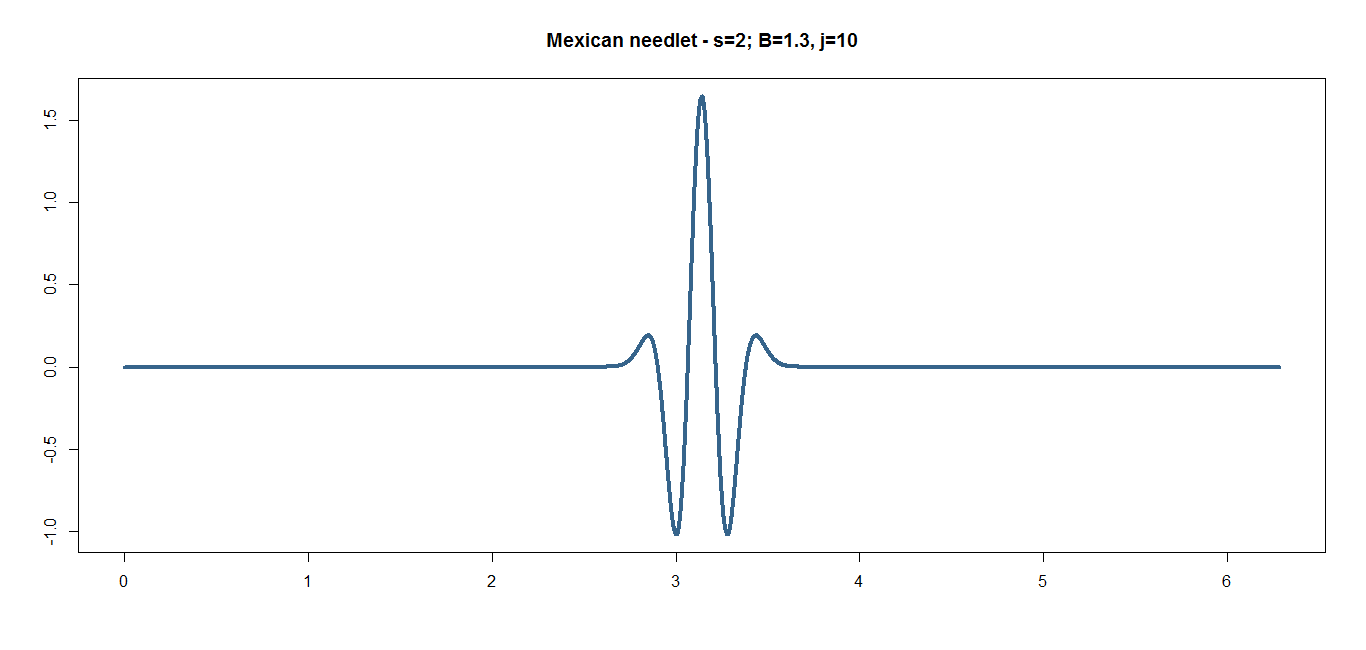}
\caption{the Mexican needlet corresponding to $s=2$, $j=10$ and $B=1.3$.}
\label{fig:2}
\end{figure}

For any $F\in $ $L^{2}\left( \mathbb{S}^{1}\right) $, the needlet
coefficient $\beta _{jq;s}\in \mathbb{C}$ associated to $\psi _{jq;s}$ is
given by%
\begin{equation}
\beta _{jq;s}:=\left\langle F,\psi _{jq;s}\right\rangle _{L^{2}\left( 
\mathbb{S}^{1}\right) }\text{ .}  \label{needletcoeff}
\end{equation}

The system $\left\{ \psi _{jq;s}\right\} $, under some regularity
conditions, describes a nearly-tight frame over $\mathbb{S}^{1}$, as proved
by the Theorem 1.1 in \cite{gm2} (for general manifold). A set of functions $%
\left\{ g_{i},i\geq 1\right\} $ over a manifold $M$ is said a frame if there
exist two positive constant $c_{1}$ and $c_{2}$ (the tightness constants)
such that for any $F\in L^{2}\left( M\right) $ 
\begin{equation*}
c_{1}\left\Vert F\right\Vert _{L^{2}\left( M\right) }^{2}\leq
\sum_{i}\left\vert \left\langle F,x_{i}\right\rangle _{L^{2}\left( M\right)
}\right\vert ^{2}\leq c_{2}\left\Vert F\right\Vert _{L^{2}\left( M\right)
}^{2}\text{ .}
\end{equation*}%
The frame is tight if $c=c_{1}=c_{2}$. If the frame is tight, it is characterized by a \textit{reconstruction formula}, i.e. for any%
 $F\in L^{2}\left( M\right) $, the following equality holds in the $L^2$-sense:

\begin{eqnarray*}
F = \frac{1}{\sqrt{c}}\sum_{i} \left\langle F,x_{i}\right\rangle _{L^{2}} x_{i} \text{ ,}
\end{eqnarray*}
which can be roughly viewed as the counterpart of the harmonic expansion in the wavelet framework. 
 As example, we recall that the standard %
spherical needlets describe a tight frame over the $d$-dimensional sphere $%
\mathbb{S}^{d}$, cfr. \cite{npw1, npw2}; a frame is nearly-tight if $%
c_{2}/c_{1}\simeq 1+\varepsilon $, $\varepsilon $ sufficiently close to $0$.
In this case, a reconstruction formula does not hold anymore, but it is possible to build a \textit{summation formula}, such that 
\begin{eqnarray*}
F = \frac{1}{\sqrt{c}}\sum_{i} \left\langle F,x_{i}\right\rangle _{L^{2}} x_{i} + B_\epsilon \text{ .}
\end{eqnarray*}
The bias $B_\epsilon$ is smaller as $\epsilon$ is closer to $0$, cfr. \cite{durastanti3}.
Following Theorem 1.1 in \cite{gm2}, we have that, fixing $%
B>1$ and $c_{0},\delta _{0}>0$ sufficiently small, there exists a constant $%
C_{0}$ as follows: \textit{(i)} for the pixel parameter $\eta \in \left(
0,1\right) $ and for $j\in \mathbb{Z}$, there exists a set of measurable
sets $\left\{ E_{jq},q=1,...,Q_{j}\right\} $, with $\lambda _{jq}\leq \eta
B^{-j}$ and for each $j$ with $\eta B^{-j}<\delta _{0}$, $\lambda _{jq}\geq
c_{0}\left( \eta B^{-j}\right) $ for $q=1,...,Q_{j}$; \textit{(ii)} it holds
that 
\begin{equation}
\left( \Lambda _{B,s}m_{B}-C_{0}\eta \right) \left\Vert F\right\Vert
_{L^{2}\left( \mathbb{S}^{1}\right) }^{2}\leq \sum_{j=-\infty }^{\infty
}\sum_{q=1}^{Q_{j}}\left\vert \beta _{jq;s}\right\vert ^{2}\leq \left(
\Lambda _{B,s}M_{B}+C_{0}\eta \right) \left\Vert F\right\Vert _{L^{2}\left( 
\mathbb{S}^{1}\right) }^{2}.  \label{tightness}
\end{equation}%
If $\left( \Lambda _{B,s}m_{B}-C_{0}\eta \right) >0$, it follows $\left\{
\psi _{jq;s}\right\} $ is a nearly tight frame, since 
\begin{equation*}
\frac{\left( \Lambda _{B,s}M_{B}+C_{0}\eta \right) }{\left( \Lambda
_{B,s}m_{B}-C_{0}\eta \right) }\sim \frac{M_{B}}{m_{B}}=1+O_{B}\left(
\left\vert B-1\right\vert ^{2}\log \left\vert B-1\right\vert \right) \text{ .%
}
\end{equation*}

\begin{remark}
Mexican needlets present some remarkable advantages if compared to the
standard needlet systems, see for instance \cite{npw1, npw2, bkmpAoSb} and
the textbook \cite{marpecbook}: first of all, they feature a stronger
concentration property in the real domain, cfr. \cite{durastanti1,
durastanti3, gm2}. Then, they do not need an exact system of cubature points
and weights but they can be built over a more general partition given by $%
\left\{ E_{jq},q=1,...,Q_{j}\right\} $, cfr. \cite{gm2}. On the other hand,
they present also some disadvantages: spherical needlets are
characterized by a compact support in the frequency domain (cfr. \cite{npw1,
npw2}), while Mexican needlets are defined over the whole frequency
range. Furthermore, as already mentioned, standard needlets describe a tight
frame and therefore they enjoy an exact reconstruction formula, which is
lacking in the Mexican needlet framework. The former issue is ``empirically''
compensated by the form of the function $w_{s}$, strongly localized
around a dominant term in the frequency domain and very close to zero out of
a very limited set of frequencies, substantially equivalent to the compact
support of the standard needlets. As far as the latter issue is concerned,
the summation formula, the counterpart of the reconstruction formula,
is characterized by a bias which is easily controlled by the user given the
nearly-tightness of the frame (cfr. \cite{durastanti3}).
\end{remark}

From now on, we will consider just positive resolution levels $j$. In order
to respect the conditions of nearly-tightness, we impose that, for $j>0,$%
\begin{equation}
Q_{j}\approx \eta ^{-1}B^{j}\text{ , }\lambda _{jq}\approx \eta B^{-j}\text{
.}  \label{areaandcardinality}
\end{equation}%
The Mexican needlets localization property can be stated as follows: for any 
$s\in \mathbb{N}$, there exists $c_{s}$ such that%
\begin{equation*}
\left\vert \psi _{jk;s}\left( \theta \right) \right\vert \leq c_{s}B^{\frac{j%
}{2}}\exp \left( -\left( \frac{B^{j}\left( \theta -x_{jk}\right) }{2}\right)
^{2}\right) \left( 1+\left( \frac{B^{j}\left( \theta -x_{jk}\right) }{2}%
\right) ^{2s}\right) \text{ ,}
\end{equation*}%
cfr. \cite{durastanti1, durastanti3, gm2}. The localization property leads
to very relevant boundedness rules on the $L^{p}$-norms: there exist $%
\widetilde{c_{p}},\widetilde{C_{p}}>0$ such that 
\begin{equation}
\widetilde{c_{p}}B^{j\left( \frac{p}{2}-1\right) }\eta ^{\frac{p}{2}}\leq
\left\Vert \psi _{jq;s}\right\Vert _{L^{p}\left( \mathbb{S}^{1}\right)
}^{p}\leq \widetilde{C_{p}}\eta ^{\frac{p}{2}}B^{j\left( \frac{p}{2}%
-1\right) }\text{ ,}  \label{normbound}
\end{equation}%
cfr. \cite{durastanti1,durastanti3}.

\subsection{Normal approximations and Stein-Malliavin bounds\label%
{secsteinmall}}

This section provides a quick overview on the asymptotic Gaussianity of
linear functionals of Poisson random measures, properly adapted to the unit
circle $\mathbb{S}^{1}$ and initially introduced in \cite{peccati1, peccati2}%
. Here, we follow strictly the analogous findings developed on the sphere $%
\mathbb{S}^{2}$ in \cite{dmp}. Further general discussions and more
technical details can be found also in \cite{npbook, peccatitaqqu}. Let us
begin this section by introducing some distances between laws of random
variables, standard in the literature, which define topologies strictly
stronger than the convergence in distribution. While the former, the
Wasserstein distance, is used in univariate case, the latter, the $d_{2}$%
-distance, is exploited in the multivariate case. Let $g\in \mathcal{C}%
\left( \mathbb{R}^{q}\right) $: its Lipschitz norm is given by%
\begin{equation*}
\left\Vert g\right\Vert _{Lip}=\sup_{x,y\in \mathbb{R}^{q},x\neq y}\frac{%
\left\vert g\left( x\right) -g\left( y\right) \right\vert }{\left\Vert
x-y\right\Vert _{\mathbb{R}^{q}}}\text{ ;}
\end{equation*}%
for $g\in \mathcal{C}^{2}\left( \mathbb{R}^{q}\right) $, let $M_{2}\left(
g\right) $ be given by%
\begin{equation*}
M_{2}\left( g\right) =\sup_{y\in \mathbb{R}^{q}}\left\Vert Hess\left(
g\left( x\right) \right) \right\Vert _{op}\text{ ,}
\end{equation*}%
where $\left\Vert \cdot \right\Vert _{op}$ denotes the operator norm.

\begin{definition}
\label{distancew}Let $X,Y$ be two random vectors with values on $\mathbb{R}%
^{q}$, $q\geq 1$, such that $\mathbb{E}\left[ \left\Vert X\right\Vert _{%
\mathbb{R}^{q}}\right] ,$ $\mathbb{E}\left[ \left\Vert Y\right\Vert _{%
\mathbb{R}^{q}}\right] <\infty $. The Wasserstein distance $d_{W}$ between
the laws of $X~$and $Y$ is given by%
\begin{equation*}
d_{W}\left( X,Y\right) =\sup_{g:\left\Vert g\right\Vert _{Lip}\leq
1}\left\vert \mathbb{E}\left[ g\left( X\right) -g\left( Y\right) \right]
\right\vert \text{ .}
\end{equation*}
\end{definition}

\begin{definition}
\label{distanced2}Let $X,Y$ be two random vectors with values on $\mathbb{R}%
^{q}$, $q\geq 1$, such that $\mathbb{E}\left[ \left\Vert X\right\Vert _{%
\mathbb{R}^{q}}\right] ,$ $\mathbb{E}\left[ \left\Vert Y\right\Vert _{%
\mathbb{R}^{q}}\right] <\infty $. The distance $d_{2}$ between the laws of $%
X~$and $Y$ is given by%
\begin{equation*}
d_{2}\left( X,Y\right) =\sup_{g:\left\Vert g\right\Vert _{Lip}\leq
1,M_{2}\left( g\right) \leq 1}\left\vert \mathbb{E}\left[ g\left( X\right)
-g\left( Y\right) \right] \right\vert \text{ .}
\end{equation*}
\end{definition}

We recall now the definition of Poisson random measures (cfr. for instance 
\cite{peccatitaqqu}).

\begin{definition}
Let $\left( \Theta ,\mathcal{A},\mu \right) $ be a $\sigma $-finite measure
space, with no-atomic $\mu $. The collection of random variables $\left\{
N\left( A\right) :A\in \mathcal{A}\right\} $ taking values on $\mathbb{Z}%
_{+}\cup \left\{ \infty \right\} $ is a Poisson random measure (PRM) on $%
\Theta $ with control (intensity) measure $\mu $ if the following conditions
hold:

\begin{enumerate}
\item For every element $A\in \mathcal{A}$, $N\left( A\right) $ has Poisson
distribution with parameter $\mu \left( A\right) $;

\item If $A_{1},A_{2},...,A_{n}$ $\in \mathcal{A}$ are pairwise disjoint,
then $N\left( A_{1}\right) ,...,N\left( A_{n}\right) $ are independent.
\end{enumerate}
\end{definition}

\begin{remark}
\label{require}In our case, we choose $\Theta =\mathbb{R}_{+}\times \mathbb{S%
}^{1}$, with $\mathcal{A}=\mathcal{F}\left( \Theta \right) $, the Borel
subsets of $\Theta $; $N$ corresponds to a Poisson random measure on $\Theta 
$, governed by the intensity $\mu =\tau \times \nu $. As far as $\tau $ is
concerned, we have that the map $R_{t}:=\tau \left( \mathds{1}_{\left[ 0,t%
\right] }\right) $ is strictly increasing and divergent as $t\rightarrow
\infty $ and $\tau \left( \left\{ 0\right\} \right) =0$. The probability
measure on the unit circle $\nu $ is absolutely continuous with respect to
the Lebesgue uniform measure, i. e. $\nu \left( d\theta \right) =F\left(
\theta \right) \rho \left( d\theta \right) $. From now on, $F$ is bounded
away from zero, i.e. there exist two constants $M_{\infty },M_{0}>0$ such
that 
\begin{subequations}
\begin{equation}
0<M_{0}=\inf_{\theta \in \mathbb{S}^{1}}F\left( \theta \right) \leq F\left(
\theta \right) \leq \sup_{\theta \in \mathbb{S}^{1}}\left\vert F\left(
\theta \right) \right\vert =M_{\infty }<\infty \text{ , }\theta \in \mathbb{S%
}^{1}\text{.}  \label{boundF}
\end{equation}%
It follows that, for any fixed $t>0$, the mapping 
\end{subequations}
\begin{equation*}
A\mapsto N_{t}\left( A\right) =:N\left( \mathds{1}_{\left[ 0,t\right]
}\times A\right) \text{ }
\end{equation*}%
corresponds to a PRM\ over $\mathbb{S}^{1}$ with control 
\begin{equation*}
\mu _{t}\left( d\theta \right) =R_{t}\times \nu \left( d\theta \right)
=R_{t}\times F\left( \theta \right) d\theta \text{ ,}
\end{equation*}%
cfr. Point (i), Remark 2.4 in \cite{dmp}. Furthermore, as stated in Point
(ii), Remark 2.4 in \cite{dmp}, given $\left\{ X_{i},i\geq 1\right\} $, the
sequence of i.i.d. random variables on $\mathbb{S}^{1}$ with distribution $%
\nu $, for any $t>0$, it holds the identity in distribution between the two
mappings $A\mapsto N_{t}\left( A\right) $ and $A\mapsto
\sum_{i=1}^{N_{t}}\delta _{X_{i}}\left( A\right) $, where $\delta _{\cdot
}\left( \cdot \right) $ is the Kronecker delta function and $N_{t}$ is an
independent Poisson random variable with intensity $R_{t}$.
\end{remark}

For any kernel $f\in L^{2}\left( \Theta ,\mu \right) \cap L^{1}\left( \Theta
,\mu \right) $, with the notation $N\left( f\right) $ and $\widetilde{N}%
\left( f\right) $ we will denote respectively the Wiener-It\^o integrals
with respect to $N$ and the corresponding compensated measure $\widetilde{N}%
\left( A\right) =N\left( A\right) -\mu \left( A\right) $, $A\in \mathcal{F}%
\left( \Theta \right) $, with the convention $N\left( A\right) -\mu \left(
A\right) =\infty $ whereas $\mu \left( A\right) =\infty $. Furthermore, the
following \textit{isometry property holds: for every }$f,g$ $\in L^{2}\left(
\Theta ,\mu \right) $, 
\begin{equation*}
\mathbb{E}\left[ \widetilde{N}\left( f\right) \widetilde{N}\left( g\right) %
\right] =\int_{\Theta }f\left( x\right) g\left( x\right) dx\text{ ,}
\end{equation*}%
More details can be found in \cite{peccatitaqqu}.

Finally, we present to rate of convergence to Gaussianity of Wiener-It\^o
integrals with respect some compensated measure $\widetilde{N}$ obtained by
the combination of the Stein's method for probabilistic approximations and
the Malliavin calculus of variations, involving random variables lying in
the first Wiener chaos of $N$. These results are here properly adapted to $%
\mathbb{S}^{1}$. The first bound, concerning normal approximation in
dimension $1$ and the Wasserstein distance, was introduced in \cite{peccati1}%
, while the second one, for the $d$-dimensional case, $d>1$, was introduced
in \cite{peccati2}.

\begin{proposition}
\label{proppeccati}Let $h\in L^{2}\left( \mathbb{S}^{1},\mu _{t}\right) \cap
L^{3}\left( \mathbb{S}^{1},\mu _{t}\right) $, let $Z\sim \mathcal{N}\left(
0,1\right) $ and let $t>0$, then the following bound holds 
\begin{subequations}
\begin{equation}
d_{W}\left( \widetilde{N}_{t}\left( h\right) ,Z\right) \leq \left\vert
1-\left\Vert h\right\Vert _{L^{2}\left( \mathbb{S}^{1},\mu _{t}\right)
}^{2}\right\vert +\int_{\mathbb{S}^{1}}\left\vert h\left( \theta \right)
\right\vert ^{3}\mu _{t}\left( d\theta \right) \text{ .}  \label{bounduni}
\end{equation}%
Therefore, if $\lim_{t\rightarrow \infty }\left\Vert h\right\Vert
_{L^{2}\left( \mathbb{S}^{1},\mu _{t}\right) }=1$ and $\lim_{t\rightarrow
\infty }\left\Vert h\right\Vert _{L^{3}\left( \mathbb{S}^{1},\mu _{t}\right)
}=0$, it holds that 
\end{subequations}
\begin{equation*}
\widetilde{N}_{t}\left( h\right) \rightarrow _{d}Z\text{ .}
\end{equation*}%
For $d>1$, let $Z_{d}\sim \mathcal{N}_{d}\left( 0,\mathcal{\Sigma }\right) $%
, where $\Sigma $ is a $d$-dimensional positive-definite covariance matrix, $%
h_{1},...,h_{d}\in L^{2}\left( \mathbb{S}^{1},\mu _{t}\right) \cap
L^{3}\left( \mathbb{S}^{1},\mu _{t}\right) $; let%
\begin{equation*}
G_{t}:=\left( \widetilde{N}_{t}\left( h_{1}\right) ,...,\widetilde{N}%
_{t}\left( h_{d}\right) \right) \text{ ,}
\end{equation*}%
associated to a covariance matrix $\mathcal{C}_{d}$ whose elements are given
by%
\begin{equation*}
\mathcal{C}_{d}\left( i_{1},i_{2}\right) =\mathbb{E}\left[ \widetilde{N}%
_{t}\left( h_{i_{1}}\right) \widetilde{N}_{t}\left( h_{i_{2}}\right) \right]
=\left\langle h_{i_{1}},h_{i_{2}}\right\rangle _{L^{2}\left( \mathbb{S}%
^{1},\mu _{t}\right) }\text{ , }i_{1},i_{2}=1,...,d\text{ .}
\end{equation*}%
Hence it holds that 
\begin{eqnarray}
d_{2}\left( G_{t},Z_{d}\right) &\leq &\left\Vert \Sigma ^{-1}\right\Vert
_{op}\left\Vert \Sigma \right\Vert _{op}^{\frac{1}{2}}\left\Vert \Sigma -%
\mathcal{C}_{d}\right\Vert _{H.S.}+\frac{\sqrt{2\pi }}{8}\left\Vert \Sigma
^{-1}\right\Vert _{op}^{\frac{3}{2}}\left\Vert \Sigma \right\Vert _{op} 
\notag \\
&&\times \sum_{i_{1},i_{2},i_{3}=1}^{d}\int_{\mathbb{S}^{1}}\left\vert
h_{i_{1}}\left( \theta \right) \right\vert \left\vert h_{i_{2}}\left( \theta
\right) \right\vert \left\vert h_{i_{3}}\left( \theta \right) \right\vert
\mu _{t}\left( d\theta \right) \text{ ,}  \label{boundmulti}
\end{eqnarray}%
where $\left\Vert \cdot \right\Vert _{op}$ and $\left\Vert \cdot \right\Vert
_{H.S.}$ denote respectively operator and Hilbert-Schmidt norms.
\end{proposition}

\section{Rates of convergence to Gaussianity of Mexican needlet
coefficients\label{secmain}}

In this section will study the asymptotic behaviour of the means of the
Mexican needlet coefficients, establishing two quantitative central limit
theorems and their explicit rates of convergence. Under the assumptions
stated in the previous section, let $\left\{ X_{i},i\geq 1\right\} $ be a
sequence of i.i.d. random variables with distribution $\nu $, independent of
the Poisson process $\widetilde{N_{t}}\left( \mathbb{S}^{1}\right) $.\
Consider the following kernel 
\begin{equation}
h_{jq;s}^{\left( R_{t}\right) }\left( \theta \right) :=\frac{\psi
_{jq;s}\left( \theta \right) }{\sqrt{R_{t}}\sigma _{jq;s}}\text{ , }\theta
\in \mathbb{S}^{1}\text{ ,}  \label{hkernel}
\end{equation}%
where $h_{jq;s}^{\left( R_{t}\right) }\in L^{1}\left( \mathbb{S}^{1},\mu
_{t}\right) \cap L^{2}\left( \mathbb{S}^{1},\mu _{t}\right) \cap L^{3}\left( 
\mathbb{S}^{1},\mu _{t}\right) $. Let $b_{jq;s}:=\mathbb{E}\left[ \psi
_{jq;s}\left( X_{1}\right) \right] $ and $\sigma _{jq;s}^{2}:=\mathbb{E}%
\left[ \psi _{jq;s}^{2}\left( X_{1}\right) \right] $. Observe that, using (%
\ref{boundF}) and (\ref{normbound}),%
\begin{equation}
0<M_{0}\eta \widetilde{c_{2}}\leq M_{0}\left\Vert \psi _{jq;s}\right\Vert
_{L^{2}\left( \mathbb{S}^{1}\right) }^{2}\leq \sigma _{jq;s}^{2}\leq
M_{\infty }\left\Vert \psi _{jq;s}\right\Vert _{L^{2}\left( \mathbb{S}%
^{1}\right) }^{2}\leq M_{\infty }\eta \widetilde{C_{2}}<\infty
\label{sigmabound}
\end{equation}%
Let us write%
\begin{equation*}
\widetilde{\beta }_{jq;s}^{\left( R_{t}\right) }:=\widetilde{N_{t}}\left(
h_{jq;s}^{\left( R_{t}\right) }\right) =\sum_{\theta \in \limfunc{supp}%
\left( N_{t}\right) }h_{jq;s}^{\left( R_{t}\right) }\left( \theta \right)
-R_{t}\int_{\mathbb{S}^{1}}h_{jq;s}^{\left( R_{t}\right) }\left( \theta
\right) \nu \left( d\theta \right) \text{ ,}
\end{equation*}%
or, in view of the Remark \ref{require}, 
\begin{equation}
\widetilde{\beta }_{jq;s}^{\left( R_{t}\right) }=\frac{\sum_{i=1}^{N_{t}%
\left( \mathbb{S}^{1}\right) }\psi _{jq;s}\left( X_{i}\right) -R_{t}b_{jq;s}%
}{\sqrt{R_{t}}\sigma _{jq;s}}\text{ ,}  \label{betabta}
\end{equation}%
It is immediate to see that $\mathbb{E}\left[ \widetilde{\beta }%
_{jq;s}^{\left( R_{t}\right) }\right] =0$, $\mathbb{E}\left[ \left( 
\widetilde{\beta }_{jq;s}^{\left( R_{t}\right) }\right) ^{2}\right] =1$. Let
us finally define the $d$-dimensional vector%
\begin{equation}
Y_{t}=\left( \widetilde{\beta }_{jq_{1};s}^{\left( R_{t}\right) },...,%
\widetilde{\beta }_{jq_{d};s}^{\left( R_{t}\right) }\right) \text{ ,}
\label{vectorbeta}
\end{equation}%
while each element of its covariance matrix is given by%
\begin{equation}
\Upsilon _{t,j,s}\left( q_{i_{1}},q_{i_{2}}\right) :=\mathbb{E}\left[ 
\widetilde{\beta }_{jq_{i_{1}};s}^{\left( R_{t}\right) }\widetilde{\beta }%
_{jq_{i_{2}};s}^{\left( R_{t}\right) }\right] \text{ , }i_{1},i_{2}=1,...,d%
\text{ .}  \label{covbeta}
\end{equation}

Let $Z\sim \mathcal{N}\left( 0,I_{d}\right) \,$, where $I_{d}$ is the $d$%
-dimensional identity matrix. Hence the following results holds.

\begin{theorem}
\label{theorempeccatiunivariate}Let $\widetilde{\beta }_{jq;s}^{\left(
R_{t}\right) }$ by given by (\ref{betabta}). there exist $C_{0}$ such that%
\begin{equation*}
d_{W}\left( \widetilde{\beta }_{jq;s}^{\left( R_{t}\right) },Z\right) \leq
C_{0}\left( B^{-j}R_{t}\right) ^{-\frac{1}{2}}\,\ \text{.}
\end{equation*}%
Furthermore, if $\left( B^{-j}R_{t}\right) ^{-\frac{1}{2}}=o_{t}\left(
1\right) $, we have $\widetilde{\beta }_{jq;s}^{\left( R_{t}\right)
}\rightarrow _{d}Z$.
\end{theorem}

\begin{theorem}
\label{theorempeccati}Let $Y_{t}$ by given by (\ref{vectorbeta}). there
exist $C_{1},C_{2}$ such that%
\begin{equation*}
d_{2}\left( Y_{t},Z_{d}\right) \leq \widetilde{C_{2}^{\prime }}\exp \left(
-B^{2j}\left( x_{jq_{1}}-x_{jq_{2}}\right) ^{2}\right) \left(
1+B^{2sj}\left( x_{jq_{1}}-x_{jq_{2}}\right) ^{2s}\right) +\widetilde{%
C_{2}^{\prime }}d_{t}\left( B^{-j}R_{t}\right) ^{-\frac{1}{2}}\,\ \text{.}
\end{equation*}%
Furthermore, if $d_{t}\left( B^{-j}R_{t}\right) ^{-\frac{1}{2}}=o_{t}\left(
1\right) $, we have $Y_{t}\rightarrow _{d}Z_{d}$.
\end{theorem}

Observe that in both the cases we have the central limit theorem the Mexican
needlet coefficients converge in distribution to the (univariate and
multivariate) Gaussian distribution when $B^{j}=B^{j_{t}}=o\left(
R_{t}\right) $.

\begin{remark}
Recall that the object $B^{-j}R_{t}$ can viewed as the effective sample size
of the Mexican needlet (cfr. Remark 4.3 in \cite{dmp}). Indeed, $B^{-j}$ can
be thought as the 'effective' scale of the wavelet $\psi _{jq,s}$, i.e. the
dimension of the region $E_{jq}$. Hence, we obtain%
\begin{equation*}
\mathbb{E}\left[ \limfunc{card}\left\{ X_{i}:d\left( X_{i},x_{jq}\right)
\leq B^{-j}\right\} \right] \simeq R_{t}\int_{d\left( X_{i},x_{jq}\right)
\leq B^{-j}}F\left( \theta \right) d\theta \text{ ,}
\end{equation*}%
where%
\begin{equation*}
M_{0}B^{-j}R_{t}\leq R_{t}\int_{d\left( X_{i},x_{jq}\right) \leq
B^{-j}}F\left( \theta \right) d\theta \leq M_{\infty }B^{-j}R_{t}\text{ .}
\end{equation*}
\end{remark}

In what follows, before concluding this section, we will prove that the
explicit bounds deduced in the Theorems \ref{theorempeccatiunivariate} and %
\ref{theorempeccati} can be also extended to the case of linear statistics
based on vector of i.i.d. observations rather than on a Poisson measure, by
paying the price of an additional factor proportional to $d\left( n\right)
n^{-\frac{1}{4}}$. Let us define the de-Poissonized vector 
\begin{equation*}
Y_{n}^{\prime }=\frac{1}{\sqrt{n}}\left( \frac{\sum_{i=1}^{n}\psi _{j\left(
n\right) q_{1};s}\left( X_{i}\right) }{\sigma _{j\left( n\right) ,q_{1};s}}%
,...,\frac{\sum_{i=1}^{n}\psi _{j\left( n\right) q_{d\left( n\right)
};s}\left( X_{i}\right) }{\sigma _{j\left( n\right) ,q_{d\left( n\right) };s}%
}\right) \text{ ,}
\end{equation*}%
and let us recall the following result from \cite{dmp} (Lemma 1.1).

\begin{proposition}
\label{depoisson}Let $R\left( n\right) =n$ and consider $X_{i},i\geq 1$, as
random variables uniformly distributed over $\mathbb{S}^{1}$. Then, there
exists a constant $M_{dP}>0$ such that for every $n$ and every Lipschitz
function $f:\mathbb{R}^{d\left( n\right) }\mapsto \mathbb{R}$, the following
inequality holds%
\begin{equation*}
\left\vert \mathbb{E}\left[ f\left( G_{n}^{\prime }\right) -f\left(
G_{n}\right) \right] \right\vert \leq M_{dP}\left\Vert f\right\Vert _{Lip}%
\frac{d\left( n\right) }{n^{\frac{1}{4}}}\text{.}
\end{equation*}
\end{proposition}

As consequence, there exists $C_{dP}>0$ such that following upper bound holds%
\begin{equation*}
d_{2}\left( G_{n},Z_{d\left( n\right) }\right) \leq C_{dP}\left( \left(
B^{-j\left( n\right) }n\right) ^{-\frac{1}{2}}+\frac{M_{dP}}{n^{\frac{1}{4}}}%
\right) d\left( n\right) \text{ .}
\end{equation*}

\section{Proofs \label{secpeccati}}

In this Section we describe extensively the proofs of the Theorems \ref%
{theorempeccatiunivariate} and \ref{theorempeccati} and of some auxiliary
results.

\subsection{Proofs of Theorems \protect\ref{theorempeccatiunivariate} and 
\protect\ref{theorempeccati}}

Here we present the exhaustive proofs of the Theorems \ref%
{theorempeccatiunivariate} \ref{theorempeccati}, obtained by using the
explicit kernel (\ref{hkernel}) in the Proposition \ref{proppeccati} and
exploiting the properties of the Mexican needlets such as (\ref{normbound}).

\begin{proof}[Proof of the Theorem \protect\ref{theorempeccatiunivariate}]
Using (\ref{hkernel}) and (\ref{sigmabound}) in (\ref{bounduni}), we have 
\begin{equation*}
1-\left\Vert h_{jq;s}^{\left( R_{t}\right) }\right\Vert _{L^{2}\left( 
\mathbb{S}^{1},\mu _{t}\right) }^{2}=0\text{ , }
\end{equation*}%
while%
\begin{eqnarray*}
\int_{\mathbb{S}^{1}}\left\vert h_{jq;s}^{\left( R_{t}\right) }\left( \theta
\right) \right\vert ^{3}\mu _{t}\left( d\theta \right) &=&\int_{\mathbb{S}%
^{1}}\frac{\left\vert \psi _{jq;s}\left( \theta \right) \right\vert ^{3}}{%
R_{t}^{\frac{3}{2}}\sigma _{jq;s}^{3}}\mu _{t}\left( d\theta \right) \\
&\leq &\frac{R_{t}^{-\frac{1}{2}}}{\left( M_{0}\eta \widetilde{c_{2}}\right)
^{\frac{3}{2}}}\int_{\mathbb{S}^{1}}\left\vert \psi _{jq;s}\left( \theta
\right) \right\vert ^{3}\nu \left( d\theta \right) \\
&\leq &\frac{R_{t}^{-\frac{1}{2}}M_{\infty }}{\left( M_{0}\eta \widetilde{%
c_{2}}\right) ^{\frac{3}{2}}}\left\Vert \psi _{jq;s}\left( \theta \right)
\right\Vert _{L^{3}\left( \mathbb{S}^{1}\right) }^{3} \\
&\leq &\frac{\widetilde{C_{3}}M_{\infty }}{\left( M_{0}\widetilde{c_{2}}%
\right) ^{\frac{3}{2}}}\left( R_{t}B^{-j}\right) ^{-\frac{1}{2}}\text{ ,}
\end{eqnarray*}%
where in the last equality we have used (\ref{normbound}).
\end{proof}

As far as the multivariate case is concerned, we obtain the following
results.

\begin{proof}[Proof of the Theorem \protect\ref{theorempeccati}]
By definition, we have 
\begin{equation*}
\left\Vert I_{d}^{-1}\right\Vert _{op}=\left\Vert I_{d}\right\Vert _{op}^{%
\frac{1}{2}}=1\text{ ,}
\end{equation*}%
while, following the Lemma \ref{lemmacovariance}, it holds that%
\begin{equation*}
\left\Vert I_{d}-\Upsilon _{t,j,s}\right\Vert _{H.S.}
\end{equation*}%
\begin{eqnarray*}
&\leq &\sqrt{\sum_{q_{1}=q_{2}}\mathbb{E}\left[ \widetilde{\beta }%
_{jq_{1};s}^{\left( R_{t}\right) }\widetilde{\beta }_{jq_{2};s}^{\left(
R_{t}\right) }\right] } \\
&\leq &d\sup_{q_{1}\neq q_{2}=1,...,d}\frac{M_{\infty }}{M_{0}\widetilde{%
c_{2}}}e^{\left( -B^{2j}\left( x_{jq_{1}}-x_{jq_{2}}\right) ^{2}\right)
}\left( 1+\left( B^{j}\left( x_{jq_{1}}-x_{jq_{2}}\right) \right)
^{2s}\right) \text{ .}
\end{eqnarray*}%
On the other hand, the Lemma \ref{lemmapeccati2} states that there exists $%
\widetilde{C}>0$ such that%
\begin{equation*}
\sum_{i_{1},i_{2},i_{3}=1}^{d}\int_{\mathbb{S}^{1}}\left\vert \psi
_{jq_{i_{1}};s}\left( \theta \right) \right\vert \left\vert \psi
_{jq_{i_{2}};s}\left( \theta \right) \right\vert \left\vert \psi
_{jq_{i_{3}};s}\left( \theta \right) \right\vert \mu _{t}\left( d\theta
\right) \leq \widetilde{C}R_{t}dB^{\frac{j}{2}}\text{ .}
\end{equation*}%
Hence, we have that%
\begin{equation*}
\sum_{i_{1},i_{2},i_{3}=1}^{d}\int_{\mathbb{S}^{1}}\left\vert
h_{jq_{i_{1}};s}^{\left( R_{t}\right) }\left( \theta \right) \right\vert
\left\vert h_{jq_{i_{2}};s}^{\left( R_{t}\right) }\left( \theta \right)
\right\vert \left\vert h_{jq_{i_{3}};s}^{\left( R_{t}\right) }\left( \theta
\right) \right\vert \mu _{t}\left( d\theta \right) 
\end{equation*}
\begin{equation*}
\leq \frac{R_{t}^{-\frac{3}{2}}}{\left( M_{0}\eta \widetilde{c_{2}}\right) ^{%
\frac{3}{2}}}\sum_{i_{1},i_{2},i_{3}=1}^{d}\int_{\mathbb{S}^{1}}\left\vert
\psi _{jq_{i_{1}};s}\left( \theta \right) \right\vert \left\vert \psi
_{jq_{i_{2}};s}\left( \theta \right) \right\vert \left\vert \psi
_{jq_{i_{3}};s}\left( \theta \right) \right\vert \mu _{t}\left( d\theta
\right) 
\end{equation*}%
\begin{equation*}
\leq \frac{d_{t}\left( R_{t}B^{-j}\right) ^{-\frac{1}{2}}\widetilde{C}%
M_{\infty }}{\left( M_{0}\widetilde{c_{2}}\right) ^{\frac{3}{2}}}\text{ ,}
\end{equation*}%
as claimed.
\end{proof}

\begin{remark}
As far as the dimension $d$ is concerned, the bound proposed in the Theorem %
\ref{theorempeccati} holds both if $d$ is fixed and if $d=d_{t}$ grows with $%
t$, see also \cite{dmp}. In the former case, the bound obtained depends only
on $\left( B^{-j_{t}}R_{t}\right) ^{-\frac{1}{2}}$, while in the latter we
obtain 
\begin{equation*}
d_{w}\left( \widetilde{\beta }_{jq;s}^{\left( R_{t}\right) },Z\right) \leq
Cd_{t}\left( B^{-j_{t}}R_{t}\right) ^{-\frac{1}{2}}\text{.}
\end{equation*}%
To attain the convergence in distribution, it suffices that $%
d_{t}=o_{t}\left( B^{-j_{t}}R_{t}\right) ^{-\frac{1}{2}}$
\end{remark}

\subsection{Auxiliary results}

This subsection includes the proofs of the auxiliary Lemmas, mainly related
to asymptotic behaviour of the covariance matrix given in (\ref{covbeta})

\begin{lemma}
\label{lemmacovariance}For any $j>0$ and $1\leq q_{1}\neq q_{2}\leq
Q_{j^{\prime }}$, there exists a constant $C_{\xi }>0$ such that%
\begin{equation*}
\left\vert \mathbb{E}\left[ \left\vert \widetilde{\beta }_{jq_{1};s}^{\left(
R_{t}\right) }\widetilde{\beta }_{jq_{2};s}^{\left( R_{t}\right)
}\right\vert \right] \right\vert \leq C_{\xi }\exp \left( -B^{2j}\left(
x_{jq_{1}}-x_{jq_{2}}\right) ^{2}\right) \left( 1+\left( B^{j}\left(
x_{jq_{1}}-x_{jq_{2}}\right) \right) ^{2s}\right) \text{ .}
\end{equation*}
\end{lemma}

\begin{proof}
We have that%
\begin{eqnarray*}
\left\vert \mathbb{E}\left[ \left\vert \widetilde{\beta }_{jq_{1};s}^{\left(
R_{t}\right) }\widetilde{\beta }_{jq_{2};s}^{\left( R_{t}\right)
}\right\vert \right] \right\vert &=&\left\vert \frac{1}{R_{t}\sigma
_{jq_{1};s}\sigma _{jq_{2};s}}\int_{\mathbb{S}^{1}}\psi _{jq_{1};s}\left(
\theta \right) \psi _{jq_{2};s}\left( \theta \right) \mu _{t}\left( d\theta
\right) \right\vert \\
&=&\frac{1}{\sigma _{jq_{1},s}\sigma _{jq_{2},s}}\left\vert \int_{\mathbb{S}%
^{1}}\psi _{jq_{1};s}\left( \theta \right) \psi _{jq_{2};s}\left( \theta
\right) F\left( \theta \right) d\theta \right\vert \\
&\leq &\frac{M_{\infty }}{\eta M_{0}\widetilde{c_{2}}}\left\langle
\left\vert \psi _{jq_{1};s}\right\vert ,\left\vert \psi
_{jq_{2};s}\right\vert \right\rangle _{L^{2}\left( \mathbb{S}^{1}\right) }%
\text{ .}
\end{eqnarray*}%
Analogously to \cite{dmp}, we split $\mathbb{S}^{1}$ into two regions:%
\begin{eqnarray*}
S_{1} &=&\left\{ \theta \in \mathbb{S}^{1}:\left( \theta -x_{jq_{1}}\right)
^{2}>\left( x_{jq_{1}}-x_{jq_{2}}\right) ^{2}/2\right\} \\
S_{2} &=&\left\{ \theta \in \mathbb{S}^{1}:\left( \theta -x_{jq_{2}}\right)
^{2}>\left( x_{jq_{1}}-x_{jq_{2}}\right) ^{2}/2\right\} \text{ .}
\end{eqnarray*}%
so that. 
\begin{equation*}
\left\langle \left\vert \psi _{jq_{1};s}\right\vert ,\left\vert \psi
_{jq_{2};s}\right\vert \right\rangle _{L^{2}\left( \mathbb{S}^{1}\right)
}\leq \int_{S_{1}}\left\vert \psi _{jq_{1};s}\left( \theta \right) \psi
_{jq_{2};s}\left( \theta \right) \right\vert d\theta +\int_{S_{2}}\left\vert
\psi _{jq_{1};s}\left( \theta \right) \psi _{jq_{2};s}\left( \theta \right)
\right\vert d\theta
\end{equation*}%
From the localization property, it follows that 
\begin{equation*}
\int_{S_{1}}\left\vert \psi _{jq_{1};s}\left( \theta \right) \psi
_{jq_{2};s}\left( \theta \right) \right\vert d\theta
\end{equation*}%
\begin{eqnarray*}
&\leq &c_{s}^{2}\eta \int_{S_{1}}\exp \left( -B^{2j}\left( \left( \theta
-x_{jq_{1}}\right) ^{2}+\left( \theta -x_{jq_{2}}\right) ^{2}\right) \right)
\\
&&\times \left( 1+\left( \frac{B^{2j}\left( \theta -x_{jq_{1}}\right) ^{2}}{2%
}\right) ^{s}\right) \left( 1+\left( \frac{B^{2j}\left( \theta
-x_{jq_{2}}\right) ^{2}}{2}\right) ^{s}\right) B^{j}d\theta
\end{eqnarray*}%
\begin{eqnarray*}
&\leq &c_{s}^{2}\eta \exp \left( -B^{2j}\left( x_{jq_{1}}-x_{jq_{2}}\right)
^{2}\right) \int_{S_{1}}\exp \left( -B^{2j}\left( \theta -x_{jq_{2}}\right)
^{2}\right) \\
&&\times \left( 1+\left( \frac{B^{2j}\left( \theta -x_{jq_{1}}\right) ^{2}}{2%
}\right) ^{s}\right) \left( 1+\left( \frac{B^{2j}\left( \theta
-x_{jq_{1}}+x_{jq_{1}}-x_{jq_{2}}\right) ^{2}}{2}\right) ^{s}\right)
B^{j}d\theta
\end{eqnarray*}%
\begin{eqnarray*}
&\leq &\left( c_{s}^{\prime }\right) ^{2}\eta \exp \left( -B^{2j}\left(
x_{jq_{1}}-x_{jq_{2}}\right) ^{2}\right) \left( 1+B^{2j}\left(
x_{jq_{1}}-x_{jq_{2}}\right) ^{2s}\right) \\
&&\int_{S_{1}}\exp \left( -B^{2j}\left( \theta -x_{jq_{2}}\right)
^{2}\right) \left( 1+\left( B^{2j}\left( \theta -x_{jq_{1}}\right)
^{2}\right) ^{s}\right) ^{2}B^{j}d\theta \text{ .}
\end{eqnarray*}%
It is immediate to see%
\begin{equation*}
\int_{S_{1}}\exp \left( -B^{2j}\left( \theta -x_{jq_{2}}\right) ^{2}\right)
\left( 1+\left( B^{2j}\left( \theta -x_{jq_{1}}\right) ^{2}\right)
^{s}\right) ^{2}B^{j}d\theta
\end{equation*}%
\begin{equation*}
\leq \int_{0}^{\infty }\exp \left( -u^{2}\right) \left( 1+u^{2s}\right)
^{2}du
\end{equation*}%
\begin{equation*}
\leq \frac{1}{2}\left( \sqrt{\pi }+2\Gamma \left( s+\frac{1}{2}\right)
+2\Gamma \left( 2s+\frac{1}{2}\right) \right) \text{ ,}
\end{equation*}%
so that%
\begin{equation*}
\int_{S_{1}}\left\vert \psi _{jq_{1};s}\left( \theta \right) \psi
_{jq_{2};s}\left( \theta \right) \right\vert d\theta \leq \left(
c_{s}^{\prime \prime }\right) ^{2}\eta \exp \left( -B^{2j}\left(
x_{jq_{1}}-x_{jq_{2}}\right) ^{2}\right) \left( 1+B^{2sj}\left(
x_{jq_{1}}-x_{jq_{2}}\right) ^{2s}\right) \text{ .}
\end{equation*}%
This concludes the proof.
\end{proof}

\begin{remark}
\label{CIRO}For $j$ sufficiently large and for $\theta \neq x_{jq}$, there
exist $\tau >s$ and $C_{\tau }>0$ such that%
\begin{equation*}
\frac{\left( 1+\left( \frac{B^{j}\left( \theta -x_{jq}\right) }{2}\right)
^{2s}\right) }{\exp \left( -\left( \frac{B^{j}\left( \theta -x_{jq}\right) }{%
2}\right) ^{2}\right) }\leq C_{\tau }\left( 1+\left( \frac{B^{j}\left(
\theta -x_{jq}\right) }{2}\right) \right) ^{-\tau }
\end{equation*}
\end{remark}

\begin{lemma}
\label{lemmapeccati2}Let $\psi _{jq;s}\left( \cdot \right) $ be given by (%
\ref{needletdef}). For $j\geq 1$, there exists $C>0$ such that%
\begin{equation*}
\sum_{q_{1},q_{2},q_{3}=1}^{d}\int_{\mathbb{S}^{1}}\left\vert \psi
_{jq_{1};s}\left( \theta \right) \right\vert \left\vert \psi
_{jq_{2};s}\left( \theta \right) \right\vert \left\vert \psi
_{jq_{3};s}\left( \theta \right) \right\vert d\theta \leq Cd\eta ^{\frac{3}{2%
}}B^{\frac{j}{2}}\text{ .}
\end{equation*}
\end{lemma}

\begin{proof}
Let $\Im \left( \theta _{0},r\right) $ the arc centered on $\theta _{0}$ of
length $r$. Hence, for any $\theta \in \Im \left( x_{jm},B^{-j}\right) $%
\begin{equation*}
\sum_{q_{1},q_{2},q_{3}=1}^{d}\int_{\mathbb{S}^{1}}\left\vert \psi
_{jq_{1};s}\left( \theta \right) \right\vert \left\vert \psi
_{jq_{2};s}\left( \theta \right) \right\vert \left\vert \psi
_{jq_{3};s}\left( \theta \right) \right\vert d\theta
\end{equation*}%
\begin{equation*}
\leq \sum_{m}\int_{\Im \left( x_{j^{\prime }m},B^{-j}\right) }\left(
\sum_{q=1}^{d}\left\vert \psi _{jq;s}\left( \theta \right) \right\vert
\right) ^{3}d\theta \text{ .}
\end{equation*}%
Observe that%
\begin{equation*}
\sum_{q=1}^{d}\left\vert \psi _{jq;s}\left( \theta \right) \right\vert
\end{equation*}%
\begin{eqnarray*}
&\leq &c_{s}\eta ^{\frac{1}{2}}B^{\frac{j}{2}}\sum_{q=1}^{d}\exp \left(
-\left( \frac{B^{j}\left( \theta -x_{jq}\right) }{2}\right) ^{2}\right)
\left( 1+\left( \frac{B^{j}\left( \theta -x_{jq}\right) }{2}\right)
^{2s}\right) \\
&\leq &c_{s}\eta ^{\frac{1}{2}}B^{\frac{j}{2}}+c_{s}\eta B^{\frac{j}{2}%
}\sum_{q:=x_{jk}\notin \Im \left( x_{jm},B^{-j}\right) }^{d}\exp \left(
-\left( \frac{B^{j}\left( \theta -x_{jq}\right) }{2}\right) ^{2}\right) \\
&&\times \left( 1+\left( \frac{B^{j}\left( \theta -x_{jq}\right) }{2}\right)
^{2s}\right) \text{ .}
\end{eqnarray*}%
Now, in view of Remark \ref{CIRO}, there exists $\tau >s$ such that 
\begin{equation*}
\sum_{q:=x_{jq}\notin \Im \left( x_{jm},B^{-j}\right) }^{d}B^{\frac{j}{2}%
}\exp \left( -\left( \frac{B^{j}\left( \theta -x_{jq}\right) }{2}\right)
^{2}\right) \left( 1+\left( \frac{B^{j}\left( \theta -x_{jq}\right) }{2}%
\right) ^{2s}\right)
\end{equation*}%
\begin{equation*}
\leq c_{\tau }\sum_{q:=x_{jq}\notin \Im \left( x_{jm},B^{-j}\right) }^{d}B^{%
\frac{j}{2}}\left( 1+\left( B^{2j}\left( \theta -x_{jq}\right) ^{2}\right)
\right) ^{-\tau }\text{ .}
\end{equation*}%
As in \cite{dmp}, by using triangle inequality, for $x_{jq}\notin \Im \left(
x_{jm},B^{-j}\right) ,$ $\theta \in \Im \left( x_{jq},B^{-j}\right) $%
\begin{equation*}
\left( x_{jq}-x_{jm}\right) ^{2}+\left( \theta -x_{jq}\right) ^{2}\geq
\left( \theta -x_{jm}\right) ^{2}\text{ ,}
\end{equation*}%
we have%
\begin{equation*}
\sum_{q:=x_{jq}\notin \Im \left( x_{jm},B^{-j}\right) }B^{\frac{j}{2}}\left(
1+\left( B^{2j}\left( \theta -x_{jq}\right) ^{2}\right) \right) ^{-\tau }
\end{equation*}%
\begin{eqnarray*}
&\leq &\sum_{q:=x_{jq}\notin \Im \left( x_{jm},B^{-j}\right) }B^{\frac{j}{2}%
}\left( B^{2j}\left( x_{jm}-x_{jq}\right) ^{2}\right) ^{-\tau } \\
&\leq &\sum_{q:=x_{jq}\notin \Im \left( x_{jm},B^{-j}\right) }\frac{c_{\tau
}^{\prime }}{\rho \left( \Im \left( x_{jm},B^{-j}\right) \right) }\int_{\Im
\left( x_{jm},B^{-j}\right) }B^{\frac{j}{2}}\left( B^{2j}\left(
x_{jm}-x_{jq}\right) ^{2}\right) ^{-\tau }dx \\
&\leq &\sum_{q:=x_{jq}\notin \Im \left( x_{jm},B^{-j}\right) }\frac{c_{\tau
}^{\prime }}{\rho \left( \Im \left( x_{jm},B^{-j}\right) \right) }\int_{\Im
\left( x_{jm},B^{-j}\right) }2^{\tau }B^{\frac{j}{2}}\left( B^{2j}\left(
x_{jm}-x\right) ^{2}\right) ^{-\tau }dx \\
&\leq &c_{\tau }^{\prime \prime }B^{\frac{j}{2}}\text{ ,}
\end{eqnarray*}%
as in \cite{dmp}, Theorem 5.5. Therefore we have 
\begin{equation}
\sum_{q=1}^{d}\left\vert \psi _{jq;s}\left( \theta \right) \right\vert \leq
C_{\tau }\eta ^{\frac{1}{2}}B^{\frac{j}{2}}\text{ ,}  \label{ciroeq1}
\end{equation}%
which leads to%
\begin{equation*}
\sum_{m}\int_{\Im \left( x_{jm},B^{-j^{\prime }}\right) }\left(
\sum_{q=1}^{d}\left\vert \psi _{jq;s}\left( \theta \right) \right\vert
\right) ^{3}d\theta \leq \widetilde{C_{\tau }}\eta ^{\frac{3}{2}}B^{\frac{3}{%
2}j}\text{ .}
\end{equation*}%
To complete the proof, we use (\ref{ciroeq1}) to have%
\begin{eqnarray*}
\int_{\mathbb{S}^{1}}\left( \sum_{q=1}^{d}\left\vert \psi _{jq;s}\left(
\theta \right) \right\vert \right) ^{3}d\theta &=&\sum_{q_{1}=1}^{d}\int_{%
\mathbb{S}^{1}}\left\vert \psi _{jq_{1};s}\left( \theta \right) \right\vert
\sum_{q_{2}=1}^{d}\left\vert \psi _{jq;s}\left( \theta \right) \right\vert
\sum_{q_{3}=1}^{d}\left\vert \psi _{jq_{3};s}\left( \theta \right)
\right\vert d\theta \\
&\leq &C_{\tau }^{2}\eta B^{j}\sum_{q_{1}=1}^{d}\int_{\mathbb{S}%
^{1}}\left\vert \psi _{jq_{1};s}\left( \theta \right) \right\vert d\theta \\
&\leq &C_{\tau }^{2}\eta B^{j}\sum_{q_{1}=1}^{d}\left\Vert \psi
_{jq_{1};s}\right\Vert _{L^{1}\left( \mathbb{S}^{1}\right) } \\
&\leq &C_{\tau }^{2}d\eta ^{\frac{3}{2}}B^{\frac{j}{2}}\text{ .}
\end{eqnarray*}
\end{proof}

\section{An application: nonparametric density estimation\label%
{secapplication}}

In this section, we will present a practical application in the framework of
nonparametric thresholding density estimation. The thresholding techniques,
introduced in the literature by \cite{donoho}, have become a successful
tool in statistics, used in many research fields, cfr. the textbooks \cite%
{WASA, tsyb}. The asymptotic result here established are related to random
vectors on the unit circle assuming the form (in the ``de-Poissonized'' case) 
\begin{equation*}
\widehat{\beta }_{jq;s}^{\left( n\right) }=\frac{1}{n}\sum_{i=1}^{n}\psi
_{jq;s}\left( X_{i}\right) \text{ . }
\end{equation*}

Consider now a set of random circular observations $\left\{ X_{i}\in \mathbb{%
S}^{1}:i=1,...,n\right\} $ with common distribution $v\left( \theta \right)
=F\left( \theta \right) d\theta $. Let us introduce the threshold $\zeta
_{jq}\left( \tau _{n}\right) :=\mathds{1}_{\left\{ \left\vert \beta
_{jq;s}\right\vert \geq \kappa \tau _{n}\right\} }$ , where $\kappa $ is a
real-valued positive constant to be chosen to set the size of the threshold
(cfr. \cite{bkmpAoSb, durastanti3}): the thresholding density estimator is
given by 
\begin{equation}
\widehat{F}\left( \theta \right)
=\sum_{j=J_{0}}^{J_{n}}\sum_{q=1}^{Q_{j}}\zeta _{jq}\left( \tau _{n}\right) 
\widehat{\beta }_{jq;s}\psi _{jq;s}\left( \theta \right) \text{ , }\theta
\in \mathbb{S}^{1}\text{ . }  \label{estimatordef}
\end{equation}%
where $\tau _{n}=\sqrt{\frac{\log n}{n}}$, as usual in the literature (see
for instance \cite{bkmpAoSb, dgm, durastanti3, WASA}). Further details on
this topic can be found in \cite{durastanti3}. Finite-sample approximations
on the distribution of the coefficients $\widehat{\beta }_{jq;s}$ can be
useful to fix an optimal value of the thresholding constant $\kappa $, by
using a plug-in procedure built as follows

\begin{enumerate}
\item Fixed \ a resolution level $j^{\ast }$, the finite-sample
approximations on the distributions of the coefficients $\widehat{\beta }%
_{j^{\ast }q;s}$ can be establish explicitly their corresponding expected
values and variances.

\item Using those informations, an optimal threshold $\kappa \tau _{n}$ can
be built.

\item Study the nonparametric density estimator $\widehat{F}$ with the
optimal threshold.
\end{enumerate}

More in details, $\tau _{n}$ depends on $R_{t}$, while $\kappa $ can be
built on the value of sample expected values and variances. In particular,
observe that for the cut-off frequency $J_{n}\equiv J_{R_{t}}$, usually
chosen such that $B^{J_{R_{t}}}=\sqrt{R_{t}/\log \left( R_{t}\right) }$, we
have that 
\begin{equation*}
d_{2}\left( \widetilde{\beta }_{jq;s}^{\left( R_{t}\right) },Z\right) \leq
O\left( \left( \log \left( R_{t}\right) \right) ^{-\frac{1}{2}}\right) 
\underset{t\rightarrow \infty }{\longrightarrow }0\text{ .}
\end{equation*}

\section{Numerical results\label{secsimulation}}

\begin{figure}[tbp]
\centering
\includegraphics[width=\textwidth]{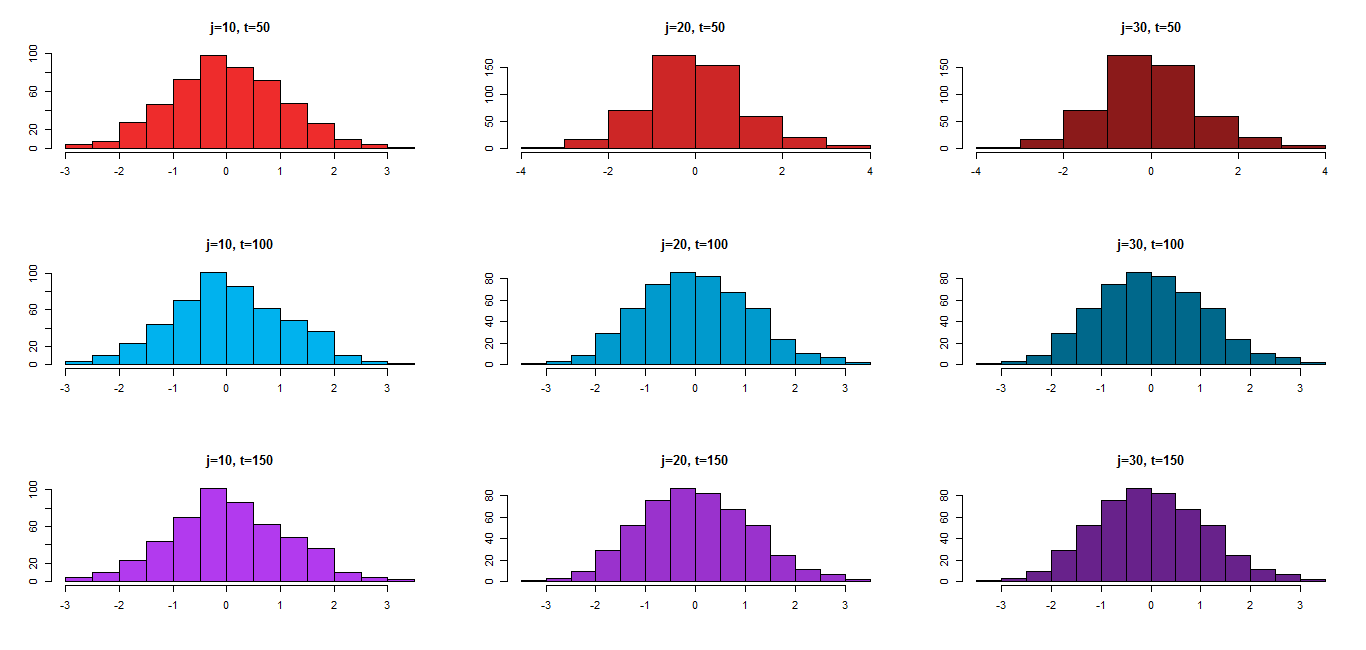}
\caption{histograms of estimates of Mexican needlet coefficients
corresponding to various choices of $j$ and $t$.}
\label{fig:3}
\end{figure}

In this section we will present some numerical evidences obtained by
simulations on CRAN\ R-code. Observe that these results, obtained in a
finite sample situation, can be considered just as a qualitative check of
the main theoretical achievements here proposed. We develop a procedure to
build Mexican coefficients, in the univariate case, according to the
following guidelines:

\begin{enumerate}
\item the distribution on $\mathbb{S}^{1}$ is uniform, such that for any $%
j,q $ 
\begin{eqnarray*}
b_{jq;s}&=& 0 ,  \notag \\
\sigma^2_{jq;s}&=& \frac{\Gamma\left(2s+\frac{1}{2}\right)}{2^{\left(2s+%
\frac{1}{2}\right)}} .  \notag
\end{eqnarray*}

\item the intensity of the Poisson process is given by $R_{t}=R\cdot t$,
with $R=10$ and $t=50,100,150$.

\item the needlets taken into account corresponds to the resolution levels: $%
j=10,20,30$, while we fixed $B=1.3$, $x_{jq}=\pi $ and $s=3$.
\end{enumerate}

\begin{table}[tbp]
\par
\begin{center}\label{tab:1}
\begin{tabular}{cccc}
\hline\hline
\multicolumn{4}{c}{Shapiro-Wilk test} \\ 
&  & $W$ & $p$-value \\ \hline\hline
& $j=10$ & 0.9962 & 0.28 \\ 
$t=50$ & $j=20$ & 0.9962 & 0.28 \\ 
& $j=30$ & 0.9960 & 0.29 \\ \hline
& $j=10$ & 0.9971 & 0.54 \\ 
$t=100$ & $j=20$ & 0.9970 & 0.53 \\ 
& $j=30$ & 0.9968 & 0.52 \\ \hline
& $j=10$ & 0.9981 & 0.85 \\ 
$t=150$ & $j=20$ & 0.9983 & 0.84 \\ 
& $j=30$ & 0.9984 & 0.86 \\ \hline\hline
\end{tabular}%
\end{center}
\caption{Result of the Shapiro-Wilk test.}
\end{table}

In Figure \ref{fig:3}, each histogram describes the normalized Mexican
coefficients built after the iteration of $N_{max}=500$ simulations,
combining $j=10,20,30$ and $t=50,100,150$. Observe that they attain fastly
the Gaussianity, as confirmed in Table 1 by the results of the
Shapiro-Wilk test. Indeed, the test statistic $W$ is closer
to $1$ for growing $t$ and it is slowing decreasing as $j$ increases.
Furthermore, $p$-values increase strongly with $t$. As a counterexample,
in Figure \ref{fig:4} we describe the distribution corresponding to the case 
$t=5,j=40$, on which the Gaussianity does not seem to be attained, as
confirmed by the Shapiro-Wilk which gives as result $W=0.9719,$ $p$-value $%
=3.207\cdot 10^{-8}$.\\

\begin{figure}[tbp]
\centering
\includegraphics[width=\textwidth]{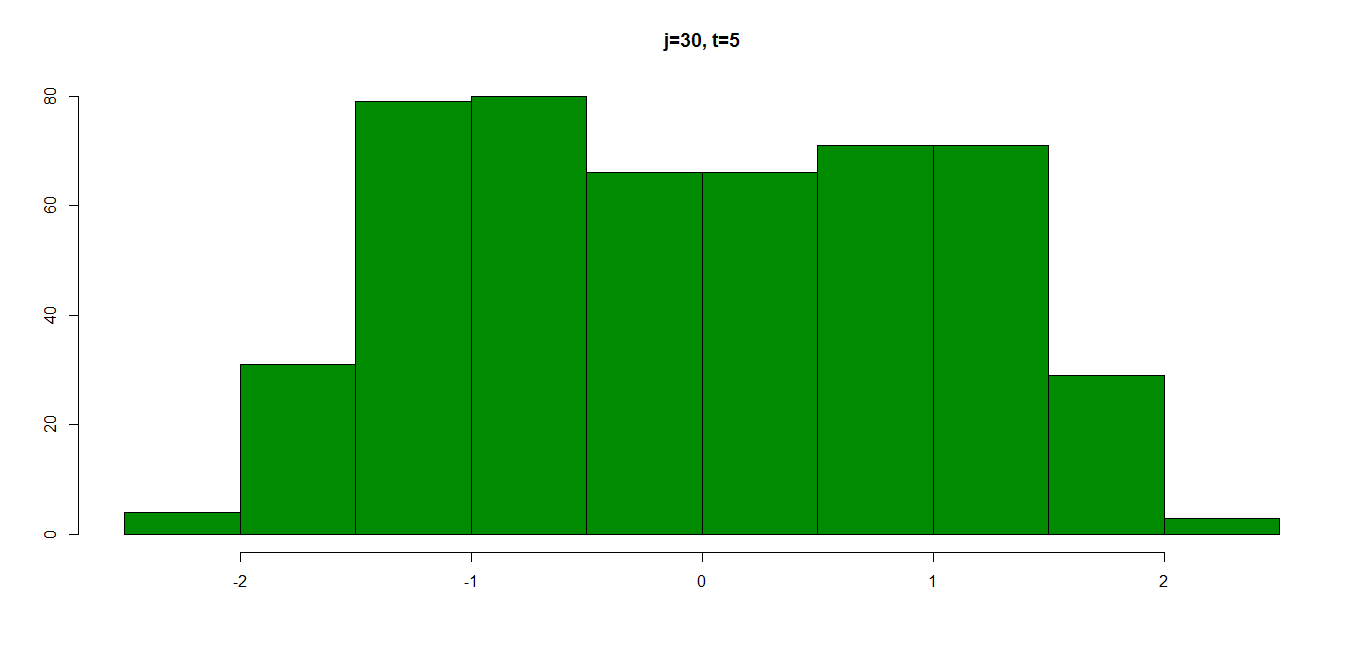}
\caption{histogram of estimates of Mexican needlet coefficients
corresponding to $j=30$ and $t=5$.}
\label{fig:4}
\end{figure}

\noindent \textbf{Acknowledgements} - The author whishes to thank D. Marinucci for the useful discussions and E. Calfa for the accurate reading.

\end{document}